\documentclass{amsart}

\usepackage{color,amssymb}%,refcheck}
\usepackage[all]{xy}

\newcommand{\w}{\omega}
\newcommand{\IN}{\mathbb N}
\newcommand{\IZ}{\mathbb Z}

\newcommand{\IG}{\mathbb G}
\newcommand{\F}{\mathcal F}

\newcommand{\Ra}{\Rightarrow}

\newtheorem{theorem}{Theorem}
\newtheorem{proposition}{Proposition}
\newtheorem{lemma}{Lemma}

\theoremstyle{definition}

\title{The Kirch space is topologically rigid}
\author{Taras Banakh, Yaryna Stelmakh and S\l awomir Turek}
\address{T.Banakh: Ivan Franko National University of Lviv (Ukraine) and Jan Kochanowski University in Kielce (Poland)}
\email{t.o.banakh@gmail.com}
\address{Ya.Stelmakh: Ivan Franko National University of Lviv (Ukraine)}
\email{yarynziya@ukr.net}
\address{S.Turek: Cardinal Stefan Wyszy\'nski University in Warsaw (Poland)}
\email{s.turek@uksw.edu.pl}
\dedicatory{Dedicated to Professor Jerzy Mioduszewski}

\begin{document}
\begin{abstract} The {\em Golomb space} (resp. the {\em Kirch space}) is the set $\IN$ of positive integers endowed with the topology generated by the base consisting of arithmetic progressions $a+b\IN_0=\{a+bn:n\ge 0\}$ where $a,b\in\IN$ and $b$ is a (square-free) number, coprime with $a$. It is known that the Golomb space (resp. the Kirch space) is connected (and locally connected). By a recent result of Banakh, Spirito and Turek, the Golomb space has trivial homeomorphism group and hence is topologically rigid. In this paper we prove the topological rigidity of the Kirch space.
\end{abstract}

\maketitle

In the AMS Meeting announcement \cite{Brown} M.~Brown introduced an amusing topology $\tau_G$ on the set $\IN$ of positive integers turning it into a connected Hausdorff space. The topology is generated by the base consisting of arithmetic progressions $a+b\IN_0:=\{a+bn:n\in\IN_0\}$ with coprime parameters $a,b\in\IN$. Here by $\IN_0=\{0\}\cup\IN$ we denote the set of non-negative integer numbers.

In \cite{SS} Brown's topology is called the {\em relatively prime integer topology}. This topology was popularized by Solomon Golomb \cite{Golomb59}, \cite{Golomb61} who observed that the classical Dirichlet theorem on primes in arithmetic progressions is equivalent to the density of the set $\Pi$ of prime numbers in the topological space $(\IN,\tau_G)$. In honour of Golomb the topological space $\IG:=(\IN,\tau_G)$ is known in General Topology as the {\em Golomb space}, see \cite{Szcz}, \cite{Szcz13}.

The problem of studying the topological structure of the Golomb space was posed to the first author (Banakh) by the third author (Turek) in 2006. In his turn, Turek learned about this problem from Jerzy Mioduszewski who listened to the lecture of Solomon Golomb on the first Toposym in 1961.

The topological structure of the Golomb space was studied by the first and third authors in \cite{BMT} and \cite{BST}. In particular, they proved that the Golomb space admits continuum many continuous self-maps but has only one homeomorphism (the identity). Topological spaces having trivial homeomorphism group are called {\em topologically rigid}. Therefore, the Golomb space is topologically rigid.

It is known that the Golomb space is connected but not locally connected. In \cite{Kirch}  Kirch introduced a topology $\tau_K\subseteq\tau_G$ turning $\IN$ into a connected and locally connected space. The {\em Kirch topology} $\tau_K$ on $\IN$ is generated by the subbase consisting of the arithmetic progressions $a+p\IN_0$ were $p$ is prime and $a\in\IN$ is not divided by $p$. The base of the Kirch topology consists of the arithmetic progressions $a+b\IN_0$ were $a,b\in\IN$ are coprime and $b$ is  square-free (i.e., $b$ is not divisible by the square of a prime number).

The main result of this note is the following rigidity theorem.

\begin{theorem}\label{t:main} The Kirch space $(\IN,\tau_K)$ is topologically rigid.
\end{theorem}

The proof of Theorem~\ref{t:main} is long and technical. It is divided into 22 lemmas.
A crucial role in the proof belongs to the superconnectedness of the Kirch space and the superconnecting poset of the Kirch space, which is defined in Section~\ref{s:poset}.

%\begin{theorem} The set $\Pi$ of prime numbers is a dense metrizable subspace of the Golomb space $\IN_\tau$.
%\end{theorem}

%For any number $x\in\IN$ let $\Pi_x$ be the set of prime divisors of $x$.

%\begin{theorem}\label{mainH} Any homeomorphism $h:\IN_\tau\to\IN_\tau$ of the Golomb space has the following properties:
%\begin{enumerate}
%\item $h(1)=1$;
%\item $h(\Pi)=\Pi$;
%\item $\Pi_{h(x)}=h(\Pi_x)$ for every $x\in \IN$.
%\end{enumerate}
%\end{theorem}

%Theorem~\ref{mainH} implies that the Golomb space is not topologically homogeneous. This answer a problem \cite{Ban}, posed by the first author on Mathoverflow.

%\begin{problem} Is the Golomb space rigid?
%\end{problem}

%We recall that a topological space $X$ is {\em rigid} if its homeomorphism group is trivial.

\section{Four classical number-theoretic results}

%For two numbers $x,y$ by $\gcd(x,y)$ we denote their greatest common divisor, and by $x\dag y$ the greatest divisor of $x$, which is coprime with $y$.

By $\Pi$ we denote the set of prime numbers. For a number $x\in\IN$ by $\Pi_x$ we denote the set of all prime divisors of $x$. Two numbers $x,y\in\IN$ are {\em coprime} iff $\Pi_x\cap\Pi_y=\emptyset$.% (which is equivalent to saying that $\gcd(x,y)=1$).

%A number $q\in\IN$ is called {\em square-free} if it is not divided by the square $p^2$ of any prime number $p$.

%For a number $x\in\IN$ and a prime number $p$ let $l_p(x)$ be the largest integer number such that $p^{l_p(x)}$ divides $x$. The function $l_p(x)$ plays the role of logarithm with base $p$. A number $x$ is square-free if and only if $l_p(x)\le 1$ for any prime number $p$.

%A function $f:X\to Y$ is called {\em finite-to-one} if for each $y\in Y$ the preimage $f^{-1}(y)$ is finite.

% For a point $x\in\IN$ by $\tau_x=\{U\in\tau\colon x\in U\}$ we denote the family of open neighborhoods of $x$ in the Kirch topology $\tau_K$ on $\IN$.

In the proof of Theorem~\ref{t:main} we will exploit the  following four known results of Number Theory. The first one is the famous Chinese Remainder Theorem (see. e.g. \cite[3.12]{J}).

\begin{theorem}[Chinese Remainder Theorem]\label{Chinese} If 
$b_1,\dots,b_n\in\IN$ are pairwise coprime numbers, then for any numbers $a_1,\dots,a_n\in\IZ$, the intersection $\bigcap_{i=1}^n(a_i+b_i\IN)$ is infinite.\end{theorem}

The second classical result is not elementary and is due to Dirichlet \cite[S.VI]{Dirichlet}, see also \cite[Ch.7]{Ap}.

\begin{theorem}[Dirichlet]\label{Dirichlet} For any coprime numbers $a,b$ the  arithmetic progression $a+b\IN$  contains a prime number. 
\end{theorem}

The third classical result is a famous theorem of Mih\u ailescu \cite{Mih}, see also \cite{Schoof}.

\begin{theorem}[Mih\u ailescu]\label{Mihailescu} If $a,b\in \big\{m^n:n,m\in \IN\setminus\{1\}\big\}$, then $|a-b|=1$ if and only if $\{a,b\}=\{2^3,3^2\}$.
\end{theorem}

The fourth classical result we use is due to Karl Zsigmondy \cite{Zsigmondy}, see also \cite[Theorem 3]{Roitman}.

\begin{theorem}[Zsigmondy]\label{Zsigmondy} For integer numbers $a,n\in\IN\setminus\{1\}$ the inclusion $\Pi_{a^n-1}\subseteq\bigcup\limits_{0<k<n}\Pi_{a^k-1}$ holds if and only if one of the following conditions is satisfied:
\begin{enumerate}
\item $n=2$ and $a=2^k-1$ for some $k\in\IN$; then $a^2-1=(a+1)(a-1)=2^k(a-1)$;
\item $n=6$ and $a=2$; then $a^n-1=2^6-1=63=3^2\times 7=(a^2-1)^2\times(a^3-1)$.
\end{enumerate}
\end{theorem}

\section{Superconnected spaces and their superconnecting posets}\label{s:poset}

In this section we discuss superconnected topological spaces and some order structures related to such spaces.

First let us introduce some notation and recall some notions.

 For a set $A$ and $n\in\w$ let $[A]^n=\{E\subseteq A:|A|=n\}$ be the family of $n$-element subsets of $A$, and $[A]^{<\w}=\bigcup_{n\in\w}[A]^n$ be the family of all finite subsets of $A$. For a function $f:X\to Y$ and a subset $A\subseteq X$ by $f[A]$ we denote the image $\{f(a):a\in A\}$ of the set $A$ under the function $f$. 
 
  For a subset $A$ of a topological space $(X,\tau)$ by $\overline{A}$ we denote the closure of $A$ in $X$. For a point $x\in X$ we denote by $\tau_x:=\{U\in\tau:x\in U\}$ the family of all open neighborhoods of $x$ in $(X,\tau)$.
A {\em poset} is an abbreviation for a partially ordered set.

A family $\F$ of subsets of a set $X$ is called a {\em filter} if
\begin{itemize}
\item $\emptyset\notin\F$;
\item for any $A,B\in\F$ we have $A\cap B\in\F$;
\item for any sets $F\subseteq E\subseteq X$ the inclusion $F\in\F$ implies $E\in\F$.
\end{itemize}

A topological space $(X,\tau)$ is called {\em superconnected} if for any $n\in\IN$ and non-empty open sets $U_1,\dots,U_n$ the intersection $\overline{U_1}\cap\dots\cap\overline{U_n}$ is non-empty. This allows us to define the filter
$$\F_\infty=\{B\subseteq X\colon \exists U_1,\dots,U_n\in\tau\setminus\{\emptyset\}\;\;(\overline{U_1}\cap\dots\cap\overline{U_n}\subseteq B)\},$$
called the {\em superconnecting filter} of $X$. 

For every finite subset $E$ of $X$ consider the subfilter
$$\F_E:=\{B\subseteq X:\textstyle{\exists (U_x)_{x\in E}\in\prod_{x\in E}\tau_x\;\;(\;\bigcap_{x\in E}\overline{U_x}\subseteq B)}\}$$
of $\F_\infty$. Here we assume that $\F_\emptyset=\{X\}$.
It is clear that for any finite sets $E\subseteq F$ in $X$ we have $\F_E\subseteq \F_F$.

The family $$\mathfrak F=\{\F_E:E\in[X]^{<\w}\}\cup\{\F_\infty\}$$ is endowed with the inclusion partial order and is called the {\em superconnecting poset} of the superconnected space $X$. The filters $\F_\emptyset$ and $\F_\infty$ are the smallest and largest elements of the poset $\mathfrak F$, respectively. 

 The following obvious lemma shows that the superconnecting poset $\mathfrak F$ is a topological invariant of the superconnected space.

\begin{proposition}\label{p:iso} For any homeomorphism $h$ of a superconnected topological space $X$, the map $$\tilde h:\mathfrak F\to\mathfrak F,\quad \tilde h\colon\F\mapsto \{h[A]\colon A\in\F\},$$ is an order isomorphism of the superconnecting poset $\mathfrak F$.
\end{proposition}

In the following sections we will study the order properties of the poset $\mathfrak F$ for the Kirch space $(\IN,\tau_K)$ and will exploit the obtained information in the proof of the topological rigidity of the Kirch space.

\section{Proof of Theorem~\ref{t:main}}

We divide the proof of Theorem~\ref{t:main} into 22 lemmas. Our first lemma  describes the closure of an arithmetic progression in the Kirch topology.

\begin{lemma}\label{basic} For any $a,b\in\IN$ the closure $\overline{a+b\IN_0}$ of the arithmetic progression $a+b\IN_0$ in the Kirch space $(\IN,\tau_K)$ is equal to
	%basic open set $a+b\IN_0$ in $\IN_\tau$ its closure
$$\IN\cap\bigcap_{p\in\Pi_b}\big(\{0,a\}+p\IZ\big).$$
\end{lemma}
% {\color{red}bla}
\begin{proof}
First we prove that $\overline{a+b\mathbb N_0}\subseteq \{0,a\}+p\IZ$ for every $p\in\Pi_b$.
Take any point $x\in\overline{a+b\IN_0}$ and assume that $x\notin p\IZ$. Then $x+p\IN_0$ is a neighborhood of $x$ and hence the intersection $(x+p\IN_0)\cap(a+b\IN_0)$ is not empty. Then there exist $u,v\in\IN_0$ such that $x+pu=a+bv$. Consequently, $x-a=bv-pu\in p\IZ$ and $x\in a+p\IZ$.

Next, take any point $x\in\IN\cap\bigcap_{p\in\Pi_b}(\{0,a\}+p\IZ)$. Given any  neighborhood $O_x$ of $x$ in $(\IN,\tau_K)$, we should prove that $O_x\cap(a+b\IN_0)\ne\emptyset$. By the definition of the Kirch topology there exists a square-free number $d\in\IN$ such that $d,x$ are coprime and $x+d\IN_0\subseteq O_x$.

If $\Pi_b\subseteq \Pi_x$, then $b,d$ are coprime and by Chinese Remainder Theorem $\emptyset \ne (x+d\IN_0)\cap (a+b\IN_0)\subseteq O_x\cap(a+b\IN_0)$.
So, we can assume $\Pi_b\setminus\Pi_x\ne\emptyset$. 
The choice of $x\in \bigcap_{p\in\Pi_b}(\{0,a\}+p\IZ)$ guarantees that $x\in \bigcap_{p\in\Pi_b\setminus \Pi_x}(a+p\IZ)=a+q\IZ$ where $q=\prod_{p\in \Pi_b\setminus\Pi_x}p$. Since the numbers $x$ and $d$ are coprime and $d$ is square-free, the greatest common divisor of $b$ and $d$ divides the number $q$. Since $x-a\in q\IZ$, the Euclides algorithm yields two numbers $u,v\in\IN_0$ such that $x-a=bu-dv$, which implies that $O_x\cap (a+b\IN_0)\supset (x+d\IN_0)\cap(a+b\IN_0)\ne\emptyset$.
\end{proof}

Lemma~\ref{basic} implies that the Kirch space $(\IN,\tau_K)$ is superconnected and hence possesses the superconnecting filter 
$$\F_\infty=\big\{F\subseteq\IN:\exists U_1,\dots,U_n\in\tau_K\setminus\{\emptyset\}\quad\textstyle\big(\bigcap\limits_{i=1}^n\overline{U_i}\subseteq F\big)\big\}$$and the superconnecting poset
$$\mathfrak F=\{\F_E:E\in[\IN]^{<\w}\}\cup\{\F_\infty\}$$
consisting  of the filters $$\F_E=\big\{F\subseteq\IN:\textstyle{\exists (U_x)_{x\in E}\in\prod_{x\in E}\tau_x\;\;\big(\bigcap_{x\in E}\overline{U_x}\subseteq F}\big)\big\}.$$ Here for a point $x\in\IN$ by $\tau_x:=\{U\in\tau_K:x\in U\}$ we denote the family of open neighborhoods of $x$ in the Kirch topology $\tau_K$.

For a nonempty finite subset $E\subseteq\IN$, let $\Pi_E=\bigcap_{x\in E}\Pi_x$ be the set of common prime divisors of numbers in the set $E$. Also let 
$$A_E=\{p\in\Pi:\exists k\in \IN\;\;(E\subset \{0,k\}+p\IZ)\}.$$
Observe that $\Pi_E\subseteq A_E$ and $A_E\ne\emptyset$ because $2\in A_E$. 
%{\color{red}Let $p\in\Pi_E$ and $x\in E$. Then $x=pl$ for some $l\in\IN$.}
If $E$ is a singleton, then $A_E=\Pi$;  
%{\color{red}Let $E=\{x\}$ and $p\in\Pi$. There is a number $k\in\IN$ such
% that $p\mid (x-k)$   }; 
if $|E|\ge 2$, then $A_E\subseteq\{1,\dots,\max E\}$.

 Indeed, assuming that $A_E$ contains some prime number $p>\max E$, we can find a number $k\in \{1,\dots,p-1\}$ such that $E\subseteq \{0,k\}+p\IZ$. Then for any distinct numbers $x,y\in E$ we get $x,y\in k+p\IZ$ and hence $x-y\in p\IZ$ which is not possible as $p>\max E\ge|x-y|$.

Let $\alpha_E\colon A_E\to\w$ be the unique function satisfying the following conditions:
\begin{itemize}
\item[\textup{(i)}] $0\le\alpha_E(p)<p$ for all $p\in A_E$;
\item[\textup{(ii)}] $E\subseteq\{0,\alpha_E(p)\}+p\IZ$ for all $p\in A_E$;
\item[\textup{(iii)}] $\alpha_E(2)=1$ and $\alpha_E(p)=0$ for all $p\in \Pi_E\setminus\{2\}$.
\end{itemize} 

%For two numbers $n\in\IN$ and $z\in\IZ$ by $(z\!\mod n)$ we denote the unique number $x\in\{0,\dots,n-1\}$ such that $z-x\in n\IZ$.

%{\color{red}
	\begin{lemma}\label{2ae} For any two-element set $E=\{x,y\}\subset \IN$ we have $A_E=\{2\}\cup \Pi_x\cup\Pi_y\cup\Pi_{x-y}$. %and 
		%$$
		%\alpha_E(p)=\begin{cases}1&\mbox{if $p=2$};\\
		%(y\!\!\mod p)&\mbox{if $p\in (\Pi_x\cup\Pi_{x-y})\setminus\{2\}$};\\
		%(x\!\!\mod p)&\mbox{if $p\in (\Pi_y\cup\Pi_{x-y})\setminus\{2\}$}.
		%\end{cases}
		%$$
\end{lemma}
	
\begin{proof} The number $p=2$ belongs to $A_E$ because $E\subset \{0,1\}+\IZ$. Each number $p\in\Pi_x$ (resp. $p\in\Pi_y$) belongs to $A_E$ because $\{x,y\}\subset\{0,y\}+p\IZ$ (resp. $\{x,y\}\subset\{0,x\}+p\IZ\}$). Each number $p\in\Pi_{x-y}$ belongs to $A_E$ because $\{x,y\}\subset x+p\IZ\subset\{0,x\}+p\IZ$. This proves that $\{2\}\cup\Pi_x\cup\Pi_y\cup\Pi_{x-y}\subseteq A_E$.

Now take any prime number $p\in A_E$ and assume that $p\notin \Pi_x\cup\Pi_y$. It follows from $\{x,y\}=E\subset\{0,\alpha_E(p)\}+p\IZ$ that $\{x,y\}\subseteq \alpha_E(p)+p\IZ$ and hence $x-y\in p\IZ$ and $p\in\Pi_{x-y}$.  
\end{proof}

The following lemma yields an arithmetic description of the filters $\F_E$.

\begin{lemma}\label{l:realization} Let $A\subset\Pi$ be a finite set such that $2\in A\ne\{2\}$ and $\alpha:A\to\IN_0$ be a function such that $\alpha(2)=1$ and $\alpha(p)\in\{0,\dots,p-1\}$ for all $p\in A\setminus\{2\}$. Let  $x$ be the product of odd  prime numbers in the set $A$ and $y$ be any number in the set $\IN\cap\bigcap_{p\in A}(\alpha(p)+p\IZ)$. Then the set $E=\{y,x,2x\}$ has $A_E=A$ and $\alpha_E=\alpha$.
\end{lemma}

\begin{proof} For every prime number $p\in A$ we have $E=\{y,x,2x\}\subset\{0,y\}+p\IZ$, which implies that $p\in A_E$. Assuming that $A_E\setminus A$ contains some prime number $p$, we conclude that $x\notin p\IZ$ and hence the inclusion $\{y,x,2x\}=E\subset \{0,\alpha_E(p)\}+p\IZ$ implies $\{x,2x\}\subset\alpha_E(p)+p\IZ$ and $x=2x-x\in p\IZ$. This contradiction shows that $A_E=A$. To show that $\alpha_E=\alpha$, take any prime number $p\in A=A_E$.
If $p=2$, then $\alpha(p)=1=\alpha_E(p)$. So, we assume that $p\ne 2$. If $\alpha(p)=0$, then $y\in \alpha(p)+p\IZ=p\IZ$ and hence $p\in \Pi_E$. In this case $\alpha_E(p)=0=\alpha(p)$. If $\alpha(p)\ne 0$, then the number $y\in\alpha(p)+p\IZ$ is not divisible by $p$ and then the inclusions $\{y,x,2x\}\subseteq\{0,\alpha(p)\}+p\IZ$ and $\{y,x,2x\}=E\subset\{0,\alpha_E(p)\}+p\IZ$ imply that $\alpha(p)=\alpha_E(p)$.
\end{proof}

% aLet $p\in A_E\setminus\{2\}$. Assume that $p\notin\Pi_x$ and observe that $\{x,y\}=E\subseteq\{0,\alpha_E(p)\}+p\IZ$ implies $x\in \alpha_E(p)+p\IZ$ and hence $\alpha_E(p)=x\!\mod p$. If $y\in p\IZ$, then $p\in\Pi_y$. Otherwise $y\in k+p\IZ$ and then $p\in\Pi_{x-y}$. 
			%Conversely, if $p\in\Pi_x\cup\Pi_{x-y}$   	
	%then $E\subseteq \{0,y\}+p\IZ$ and similarly if $p\in\Pi_y$ then $E\subseteq \{0,x\}+p\IZ$. 
	%In both of the above cases $p\in A_E$.
%\end{proof}%}

%\begin{example}\label{ex:1x} For any number $x>1$ and the set $E=\{1,x\}$ we have 
%$$\Pi_E=\emptyset,\;\;A_E=\Pi_x\cup \Pi_{x-1}\;\mbox{ and }\alpha_E(A_E)=\{1\}.$$
%\end{example}

%\begin{example}\label{ex:2o} For any odd number $x>2$ and the set $E=\{2,x\}$ we have 
%$$\Pi_E=\emptyset,\;\;A_E=\{2\}\cup \Pi_x\cup \Pi_{x-2}\;\mbox{ \ and \ }\alpha_E(2)=1, \;\alpha_E(A_E\setminus\{2\})=\{2\}.$$
%\end{example}

%\begin{example}\label{ex:2e} For any even number $x>2$ and the set $E=\{2,x\}$ we have 
%$$\Pi_E=\{2\},\;\;A_E=\Pi_x\cup \Pi_{x-2}\setminus \{2\}\mbox{ \ and \ }\alpha_E(A_E)\subseteq \{2\}.$$
%\end{example}

\begin{lemma}\label{l:Kirch} For any finite subset $E\subseteq\IN$ with $|E|\ge 2$ we have
$$\F_E=\big\{B\subseteq\IN\colon \exists L\in[\Pi\setminus A_E]^{<\w}\quad\bigcap_{p\in L}p\IN\cap\bigcap_{p\in A_E\setminus\Pi_E}(\{0,\alpha_E(p)\}+p\IZ)\subseteq B\big\}.$$Here we assume that $\bigcap_{p\in\emptyset}p\IN=\IN$.
\end{lemma}

\begin{proof} It suffices to verify  two properties:
\begin{enumerate}
\item for any $(U_x)_{x\in E}\in\prod_{x\in E}\tau_x$ there exists a finite set $L\subseteq \Pi\setminus A_E$ such that
$$\bigcap_{p\in L}p\IN\cap\bigcap_{p\in A_E\setminus\Pi_E}(\{0,\alpha_E(p)\}+p\IZ)\subseteq \bigcap_{x\in E}\overline{U_x};$$
\item for any finite set $L\subseteq\Pi\setminus A_E$ there exists a sequence of neighborhoods $(U_x)_{x\in E}\in\prod_{x\in E}\tau_x$ such that
$$\bigcap_{x\in E}\overline{U_x}\subseteq \bigcap_{p\in L}p\IN\cap\bigcap_{p\in A_E\setminus\Pi_E}(\{0,\alpha_E(p)\}+p\IZ).$$
\end{enumerate}

1. Given a sequence of neighborhoods $(U_x)_{x\in E}\in\prod_{x\in E}\tau_x$, for every $x\in E$ find a square-free number $q_x>x$ such that $\Pi_{q_x}\cap \Pi_x=\emptyset$ and $x+q_x\IN_0\subseteq U_x$. We claim that the finite set $L=\bigcup_{x\in E}\Pi_{q_x}\setminus A_E$ has the required property. Given any number $z\in \bigcap\limits_{p\in L}p\IN\cap\bigcap\limits_{p\in A_E\setminus\Pi_E}(\{0,\alpha_E(p)\}+p\IZ)$, we should prove that $z\in\overline{U_x}$ for every $x\in E$. By  Lemma~\ref{basic},
$$\IN\cap \bigcap_{p\in\Pi_{q_x}}(\{0,x\}+p\IZ)=\overline{ (x+q_x\IN_0)}\subseteq\overline{U_x}.$$ So, it suffices to show that $z\in \{0,x\}+p\IZ$ for any $p\in \Pi_{q_x}$. Since the numbers $x$ and $q_x$ are coprime, $p\notin\Pi_x$ and hence $p\notin\Pi_E$. If $p\notin A_E$, then $p\in \Pi_{q_x}\setminus  A_E\subseteq L$ and hence $z\in p\IN\subseteq \{0,x\}+p\IZ$. If $p\in A_E$, then $x\in E\subseteq\{0,\alpha_E(p)\}+p\IZ$ and $x\in\alpha_E(p)+p\IZ$ (as $p\notin\Pi_x$). Then $x+p\IZ=\alpha_E(p)+p\IZ$ and  
$z\in \{0,\alpha_E(p)\}+p\IZ=\{0,x\}+p\IZ.$ 
\smallskip

2. Fix any finite set $L\subseteq \Pi\setminus  A_E$. For every $x\in E$ consider the neighborhood $U_x=\bigcap_{p\in L\cup A_E\setminus\Pi_x}(x+p\IN_0)$ of $x$ in the Kirch topology. By Lemma~\ref{basic},
$$\overline{U_x}=\IN\cap\bigcap_{p\in L\cup A_E\setminus \Pi_x}(\{0,x\}+p\IZ).$$

Given any number $z\in\bigcap_{x\in E} \overline{U_x}$, we should show that $z\in  \bigcap\limits_{p\in L}p\IN\cap\bigcap\limits_{p\in A_E\setminus \Pi_E}(\{0,\alpha_E(p)\}+p\IZ)$. 
This will follow as soon as we check that $z\in p\IN$ for all $p\in L$ and $z\in\{0,\alpha_E(p)\}+p\IZ$ for all $p\in A_E\setminus \Pi_E$.

Given any $p\in A_E\setminus\Pi_E$, we can find a point $x\in E\setminus p\IZ$ and observe that $x\in E\subseteq\{0,\alpha_E(p)\}+p\IZ$. Then $z\in \overline{U_x}\subseteq\overline{x+p\IN_0}\subseteq \{0,x\}+p\IZ=\{0,\alpha_E(p)\}+p\IZ$. 

Now take any prime number $p\in L$. Since $L\cap  A_E=\emptyset$, we conclude that $E\not\subseteq p\IZ$. So, we can fix a number $x\in E\setminus p\IZ$. Taking into account that $p\notin A_E$, we conclude that $E\not\subseteq \{0,x\}+p\IZ$ and hence there exists a number $y\in E$ such that $p\IZ\ne y+p\IZ\ne x+p\IZ$. Then 
$$z\in\overline{U_x}\cap\overline{U_y}\subseteq(\{0,x\}+p\IZ)\cap(\{0,y\}+p\IZ)=p\IZ.$$ 
\end{proof}

We shall use Lemma~\ref{l:Kirch} for an arithmetic characterization of the partial order of the superconnecting poset $\mathfrak F$ of the Kirch space.

\begin{lemma}\label{l:wo2} For two finite subsets $E,F\subseteq \IN$ with $\min\{|E|,|F|\}\ge2$ we have $\F_E\subseteq \F_F$ if and only if $$A_F\subseteq A_E,\;\; \Pi_F\setminus\{2\}\subseteq \Pi_E \mbox{ \  and \ }\alpha_E{\restriction}A_F\setminus\Pi_E=\alpha_F{\restriction}A_F\setminus\Pi_E.$$
\end{lemma}

\begin{proof} To prove the ``only if'' part, assume that $\F_E\subseteq\F_F$.  By Lemma~\ref{l:Kirch}, the set $$\bigcap_{p\in A_F\setminus A_E}p\IN\cap\bigcap_{p\in A_E\setminus\Pi_E}(\{0,\alpha_E(p)\}+p\IZ)$$ belongs to the filter $\F_E\subseteq\F_F$. 
By Lemma~\ref{l:Kirch}, there exists a finite set $L\subseteq\Pi\setminus A_F$ such that  
\begin{equation}\label{eq1}
 \bigcap_{p\in L}p\IN\cap\bigcap_{p\in A_F\setminus\Pi_F}(\{0,\alpha_F(p)\}+p\IZ)\subseteq \bigcap_{p\in A_F\setminus A_E}p\IN\cap \bigcap_{p\in A_E\setminus\Pi_E}(\{0,\alpha_E(p)\}+p\IZ).
\end{equation}
This inclusion combined with the Chinese Remainder Theorem~\ref{Chinese} implies $$A_F\setminus A_E\subseteq L\subset \Pi\setminus A_F,\;\;A_E\setminus (\Pi_E\cup\{2\})\subseteq L\cup (A_F\setminus\Pi_F)\mbox{ \  and $\alpha_E(p)=\alpha_F(p)$ for any $p\in (A_F\setminus\Pi_F)\cap (A_E\setminus\Pi_E)$,}$$
and 
\begin{equation}\label{eq2}
A_F\subseteq A_E,\;\;\Pi_F\setminus\{2\}\subseteq\Pi_E\mbox{ \  and \ }\alpha_E{\restriction}A_F\setminus\Pi_E=\alpha_F{\restriction}A_F\setminus\Pi_E.
\end{equation}

To prove the ``if'' part, assume that the condition (\ref{eq2}) holds. To prove that $\F_E\subseteq\F_F$, fix any set $\Omega\in\F_E$ and using Lemma~\ref{l:Kirch}, find a finite set $L\subseteq \Pi\setminus A_E$ such that
$$\bigcap_{p\in L}p\IN\cap\bigcap_{p\in A_E\setminus\Pi_E}(\{0,\alpha_E(p)\}+p\IZ)\subseteq\Omega.$$
Consider the finite set $\Lambda=(L\cup A_E)\setminus A_F=L\cup (A_E\setminus A_F)\supseteq L$ and observe that the condition (\ref{eq2}) implies the inclusion
\begin{equation}\label{eq3}\F_F\ni 
\bigcap_{p\in\Lambda}p\IN\cap\bigcap_{p\in A_F\setminus\Pi_F}(\{0,\alpha_F(p)\}+p\IZ)\subseteq   \bigcap_{p\in L}p\IN\cap\bigcap_{p\in A_E\setminus \Pi_E}(\{0,\alpha_E(p)\}+p\IZ)\subseteq\Omega,
\end{equation}
yielding $\Omega\in\F_F$.
\end{proof}

\begin{lemma}\label{l:wo1} For two nonempty subsets $E,F\subseteq \IN$ with $\min\{|E|,|F|\}=1$ the relation $\F_E\subseteq\F_F$ holds if and only if $|E|=1$ and $E\subseteq F$.
\end{lemma}

\begin{proof} The ``if'' part is trivial. To prove the ``only if'' part, assume that $\F_E\subseteq\F_F$. First we prove that $|E|=1$. Assuming that $|E|>1$ and taking into account that $\min\{|E|,|F|\}=1$, we conclude that $|F|=1$. Choose a prime number $p>\max(E\cup F)$. Since $\bigcap_{y\in E}\overline{y+p\IN_0}\in \F_E\subseteq\F_F$, for the unique number $x$ in the set $F$, there exists a square-free number $d$ such that $\Pi_d\cap\Pi_x=\emptyset$ and $\overline{x+dp\IN_0}\subseteq\bigcap_{y\in E}\overline{y+p\IN_0}$. By Lemma~\ref{basic},
$$x+qp\IN\subseteq \overline{x+dp\IN_0}\subseteq\bigcap_{y\in E}\overline{y+p\IN_0}=\bigcap_{y\in E}(\{0,y\}+p\IN_0)=p\IN_0.$$
The latter equality follows from $p>\max E$ and $|E|>1$. Then $x+dp\IN\subseteq p\IN_0$ implies $x\in p\IN_0$, which contradicts the choice of $p>\max (E\cup F)\ge x$. This contradiction shows that $|E|=1$. Let $z$ be the unique element of the set $E$.

It remains to prove that $z\in F$. To derive a contradiction, assume that $z\notin F$. Take any odd prime number $p>\max(E\cup F)$ and consider the set $\{0,z\}+p\IN_0=\overline{z+p\IN_0}\in\F_E\subseteq\F_F$. By the definition of the filter $\F_F$, for every $x\in F$ there exists a square-free number $d_x$ such that $\Pi_{d_x}\cap\Pi_x=\emptyset$ and $$\bigcap_{x\in F}\overline{x+d_x\IN_0}\subseteq \overline{z+p\IN_0}=\{0,z\}+p\IN_0.$$ Consider the set $P=\bigcup_{x\in F}\Pi_{d_x}$. 
 If $p\in \Pi_{d_x}$ for some $x\in F$, we can use the Chinese Remainder Theorem~\ref{Chinese} and find a number $$c\in (x+p\IN_0)\cap \bigcap_{q\in P\setminus\{p\}}q\IN\subseteq \bigcap_{y\in F}\overline{y+d_y\IN_0}\subseteq \{0,z\}+p\IN_0.$$
 Taking into account that $x$ is not divisible by $p$, we conclude that $c\in (x+p\IZ)\cap(z+p\IZ)$ and hence $x-z\in p\IZ$, which contradicts the choice of $p>\max(E\cup F)$. This contradiction shows that $p\notin P$. Since $p\ge 3$, we can find a number $z'\notin\{0,z\}+p\IZ$ and using the Chinese Remainder Theorem~\ref{Chinese}, find a number
 $$u\in (z'+p\IN_0)\cap\bigcap_{q\in P}q\IN\subseteq \bigcap_{x\in F}\overline{x+d_x\IN_0}\subseteq\{0,z\}+p\IZ,$$ which is a desired contradiction showing that $E\subseteq F$.   
\end{proof}

%Lemma~\ref{l:wo2} implies the following important fact.

%\begin{lemma}\label{l:wo} For any $\F\in \mathfrak F$ with $|E|\ge 2$ the upper set ${\uparrow}\F=\{\mathcal E\in\mathfrak F:\F\subseteq \mathcal E\}$ is finite.
%\end{lemma}

%The following problem seems to be of crucial importance.

%\begin{problem} Fiven a finite set $E\subset\IN$ describe the order structure of the upper set ${\uparrow}\F_E$ in the poset $\mathfrak F$.
%\end{problem} 

As we know, the largest element of the superconnecting poset $\mathfrak F$ is the superconnecting filter $\F_\infty$. This filter can be characterized as follows. 

\begin{lemma}\label{filter0} The superconnecting filter $\F_\infty$ of the Kirch space  is generated by the base consisting of the sets $q\IN$ for odd square-free numbers  $q\in\IN$, i.e.
$$\F_\infty=\{B\subseteq \IN\colon (\exists q)(q\text{ is an odd square-free})\land q\IN\subseteq B\}.$$
\end{lemma}

\begin{proof} Lemma~\ref{basic} implies that each element $F\in\F_\infty$ contains the set $q\IN$ for some odd square-free number $q$. Conversely, let $q$ be an odd square-free number. Then $U_1=1+q\IN_0,U_2=2+q\IN_0\in\tau_K$. By Lemma~\ref{basic} we have  
		$$\overline{U_1}\cap\overline{U_2}=\IN\cap\bigcap_{p\in\Pi_q}(\{0,1\}+p\IZ)\cap(\{0,2\}+p\IZ)=\IN\cap\bigcap_{p\in\Pi_q}p\IZ=  q\IN.$$
		Hence $q\IN\in\F_\infty$. 
\end{proof} 

\begin{lemma}\label{l:2max} For a nonempty subset $E\subseteq \IN$ the following conditions are equivalent:
\begin{enumerate}
\item $\F_E=\F_\infty$;
\item $A_E=\{2\}$.
\end{enumerate}
If  $|E|=2$, then the conditions \textup{(1), (2)} are equivalent to
\begin{enumerate}
\item[(3)] $E=\{2^n,2^{n+1}\}$ for some $n\in\w$.
\end{enumerate}
\end{lemma}

\begin{proof} %$(2)\Ra(1)$ Assume that $E=\{2^n,2^{n+1}\}$ for some $n\in\w$. It follows that $A_E=\{2\}$ and $\alpha_E(A_E)=\{1\}$. Then $\{0,\alpha_E(2)\}+2\IZ=\IZ$ and by Lemmas~\ref{l:Kirch} and \ref{filter0},  
%$$\F_D=\{F\subseteq \IN:\exists L\in[\Pi\setminus\{2\}]^{<\w}\;\;\bigcap_{p\in L}p\IN\subseteq F\}=\F_\infty.$$
%\smallskip

%$(1)\Ra(2)$ 
%{\color{red} 
$(1)\Rightarrow(2)$: Assume $\F_{E}=\F_\infty$. Consider $F=\{1,2\}$. 
It is clear that $A_F=\{2\}$ and $\Pi_F=\emptyset$. Thus $A_F\subseteq A_E$, $\Pi_F{\setminus}\{2\}\subseteq \Pi_E$ and $\alpha_F{\restriction}A_F{\setminus} \Pi_E=\alpha_E{\restriction} A_F{\setminus} \Pi_E$. Lemma~\ref{l:wo2} implies $\F_{E}\subseteq \F_F$. Since $\F_{E}=\F_\infty$ is the largest element of $\mathfrak{F}$ we get $\F_E=\F_F$. By using again Lemma~\ref{l:wo2} we get $A_E\subseteq A_F$ which implies that $A_E = \{2\}$.
\vskip3pt	
	
$(2)\Rightarrow(1)$:
%\footnote{\color{blue} What about the implication $(1)\Rightarrow(2)$? 
%At the moment it is established only under the additional condition $|E|=2$?} 
If $A_E=\{2\}$, then by the Lemma~\ref{l:Kirch}, the filter $\F_{E}$ is generated by the base consisting of the sets $q\IN$ for an odd square-free number $q\in\IN$. Therefore $\F_E=\F_\infty$ by the Lemma~\ref{filter0}.
\smallskip	
	
If $|E|=2$, then the equivalence $(2)\Leftrightarrow(3)$ follows from Lemma~\ref{2ae}.
\end{proof}

\begin{lemma}\label{l:2fix} For every $n\in\w$, the number $2^n$ is a fixed point of any homeomorphism $h$ of the Kirch space.
\end{lemma}

\begin{proof} Consider the graph $\Gamma_2=(V_2,\mathcal E)$ with set of vertices $V_2=\{2^n:n\in\w\}$ and set of edges $\mathcal E=\big\{\{2^n,2^{n+1}\}:n\in\w\big\}$. %Obserev that the graph $\Gamma_2$ is rigid in the sense that it admits a unique isomorphism. This follows from the observation that $1$ is a unique vertex of order 1 in $\Gamma_2$. Consequently, $1$ is a fixed point of any by $2$ is a unique vertex connected by an edge with the 

By Lemma~\ref{l:2max}, for every edge $E\in\mathcal E$ of the graph $\mathcal E$ we have $\F_E=\F_\infty$ and hence $\F_{h[E]}=\tilde h(\F_E)=\tilde h(\F_\infty)=\F_{\infty}$ by the topological invariance of the filter $\F_\infty$. Applying Lemma~\ref{l:2max} once more, we conclude that $h[E]\in\mathcal E$. The same argument applied to the homeomorphism $h^{-1}$ ensures that $\tilde h^{-1}[E]\in\mathcal E$ for any $E\in\mathcal E$. This means that $\tilde h$ induces an isomorphism of the graph $\Gamma_2$. Now observe that the number $2^0=1$ is a unique vertex of the graph $\Gamma_2$ that has order 1. This graph-theoretic property of the vertex $2^0$ in $\Gamma_2$ ensures that $h(2^0)=2^0$. Next, observe that $2^1$ is a unique vertex of $\Gamma_2$ that is connected with $2^0$ and hence $h(2^1)=2^1$. Proceeding by induction, we can show that $h(2^n)=2^n$ for all $n\in\w$.
\end{proof}

In the following lemmas by $\mathfrak F'$ we denote the set  of maximal elements of the poset $\mathfrak F\setminus\{\F_\infty\}$.

%By $2^{<\w}$ we shall denote the set $\{2^n:n\in\w\}$ of powers of $2$.

\begin{lemma}\label{l:max} For a finite subset $E\subseteq \IN$ the filter $\F_E$ belongs to the family $\mathfrak F'$ if and only if there exists an odd prime number $p\notin\Pi_E$ such that $A_E=\{2,p\}$.
\end{lemma}

\begin{proof} To prove the ``if'' part, assume that $A_E=\{2,p\}$ and $p\notin\Pi_E$ for some odd prime number $p$. By Lemma~\ref{l:2max}, $\F_E\ne\F_\infty$. To show that the filter $\F_E$ is maximal in $\mathfrak F\setminus\{\F_\infty\}$, take any finite set $F\subset\IN$ such that $\F_E\subseteq\F_F\ne\F_\infty$. By Lemmas~\ref{l:wo2} and \ref{l:2max}, $\{2\}\ne A_F\subseteq A_E=\{2,p\}$, $\Pi_F\subseteq \Pi_E\cup\{2\}=\{2\}$, and $\alpha_F{\restriction}A_F\setminus \Pi_E=\alpha_E{\restriction}A_F\setminus \Pi_E$. It follows that $A_F=\{2,p\}=A_E$, $\Pi_F\cup\{2\}=\Pi_E\cup\{2\}$ and $\alpha_F=\alpha_E$. Applying Lemma~\ref{l:wo2}, we conclude that $\F_E=\F_F$, which means that the filter $\F_E$ is a maximal element of the poset $\mathcal F\setminus\{\F_\infty\}$.
\smallskip

To prove the ``only if'' part, assume that $\F_E\in\mathfrak F'$. By Lemma~\ref{l:2max}, $A_E\ne\{2\}$ and hence there exists an odd prime number $p\in A_E$. 
We claim that $p\notin\Pi_E$. To derive a contradiction, assume that $p\in \Pi_E$ and consider the sets $F=\{p,2p\}$ and $G=\{1,p,2p\}$. By Lemma~\ref{2ae}, $A_F=A_G=\{2,p\}$, $\Pi_F=\{p\}$, and $\Pi_G=\emptyset$. Taking into account that  $F\subset G$, $A_F=\{2,p\}\subseteq A_E$, $\Pi_F\setminus\{2\}=\{p\}\subseteq \Pi_E$ and $A_F\setminus\Pi_E\subseteq\{2\}$, we can apply Lemmas~\ref{l:wo2}, \ref{l:2max} and conclude that $\F_E\subseteq \F_F\subseteq\F_G\ne\F_\infty$. The maximality of $\F_E$ implies $\F_E=\F_F=\F_G$. By Lemma~\ref{l:wo2}, the equality $\F_G=\F_F$ implies $p\in\Pi_F\setminus\{2\}\subseteq\Pi_G=\emptyset$, which is a contradiction showing that $p\notin\Pi_E$.

Now consider the set $H=\{\alpha_E(p),p,2p\}$ and observe that $A_H=\{2,p\}$, $\Pi_H=\emptyset$  and $\alpha_H(p)=\alpha_E(p)$. Lemmas~\ref{l:wo2} and \ref{l:2max} guarantee that $\F_E\subseteq\F_H\ne\F_\infty$. By the maximality of $\F_E$, we have $\F_E=\F_H$. Applying Lemma~\ref{l:wo2} once more, we conclude that $A_E=A_H=\{2,p\}$. 
\end{proof}

Lemma~\ref{l:max} implies the following description of the set $\mathfrak F'$.

\begin{lemma} $\mathfrak F'=\{\F_{\{a,p,2p\}}:p\in\Pi\setminus\{2\},\;\;a\in\{1,\dots,p-1\}\}$.
\end{lemma}

Let $\mathfrak F''$ be the set of maximal elements of the poset $\mathfrak F\setminus(\mathfrak F'\cup\{\mathcal F_\infty\})$

\begin{lemma}\label{efbis}
	 For a finite set $E\subset\IN$, the filter $\mathcal F_E$ belongs to the family $\mathfrak F''$ if and only if one of the following conditions holds:
\begin{enumerate}
\item there exists an odd prime number $p$ such that $p\in \Pi_E$ and $A_E=\{2,p\}$;
\item there are two distinct odd prime numbers $p,q$ such that $A_E=\{2,p,q\}$ and $\Pi_E\subseteq\{2\}$.
\end{enumerate}
\end{lemma} 

\begin{proof} To prove the ``only if'' part, assume that $\F_E\in\mathfrak{F}''$. By Lemma~\ref{l:2max}, there is an odd prime number $p\in A_E$. If  $A_E=\{2,p\}$, then $p\in\Pi_E$ by Lemma~\ref{l:max}, and  condition (1) is satisfied. So, we assume that $\{2,p\}\ne A_E$ and find an odd prime number $q\in A_E\setminus\{2,p\}$. By Lemma~\ref{l:realization},  there is a number $x\in\IN$ such that for the set $F=\{x,pq,2pq\}$ we have $A_F=\{2,p,q\}$, $\Pi_F=\emptyset$, $\alpha_F(p)=a$ and $\alpha_F(q)=b$. Then $\F_E\subseteq \F_F$ by Lemma~\ref{l:wo2}, and $\F_F\in\mathfrak{F}\setminus(\mathfrak{F}'\cup\{\F_\infty\})$ by Lemma~\ref{l:max}. Now the maximality of the filter $\F_E$ implies that $\F_E=\F_F$ and hence $A_E=A_F=\{2,p,q\}$ and $\Pi_E\subset \Pi_F\cup\{2\}=\{2\}$, see Lemma~\ref{l:wo2}.
\smallskip

To prove the ``if'' part, we consider two cases. First we assume that $A_E=\{2,p\}$ for some $p\in \Pi_E$.	By Lemmas~\ref{l:2max}  and \ref{l:max}, $\F_{E}\in \mathfrak{F}\setminus(\{\F_\infty\}\cup \mathfrak{F}')$. To prove that $\F_E$ is a maximal element of 	$\mathfrak{F}\setminus(\{\F_\infty\}\cup \mathfrak{F}')$, take any finite set $F\subseteq\IN$ such that $\F_E\subseteq\F_F\in\mathfrak{F}\setminus(\{\F_\infty\}\cup \mathfrak{F}')$. Lemma~\ref{l:wo1} implies that $\min\{|E|,|F|\}\ge 2$ and then by Lemmas~\ref{l:wo2} and \ref{l:max}, we have $A_F=\{2,p\}$, $\Pi_F\setminus\{2\}\subseteq \{p\}$ and 
	$\alpha_E{\restriction}A_F\setminus\{p\}=\alpha_F{\restriction}A_F\setminus\{p\}$. Now notice that $p \in \Pi_F$ since otherwise $\F_F\in\mathfrak{F}'$ by Lemma~\ref{l:max}. By using again Lemma~\ref{l:wo2}  we get $\F_F=\F_E$ which means that $\F_E\in\mathfrak{F}''$. 		
	
	Now assume that there are two distinct odd prime numbers $p,q$ such that $A_E=\{2,p,q\}$ and $\Pi_E\subseteq\{2\}$. By Lemmas~\ref{l:2max}  and \ref{l:max}, $\F_{E}\in \mathfrak{F}\setminus(\{\F_\infty\}\cup \mathfrak{F}')$. To prove that $\F_E$ is a maximal element of  $\mathfrak{F}\setminus(\{\F_\infty\}\cup \mathfrak{F}')$, take any finite set $F\subseteq\IN$ such that $\F_E\subseteq \F_F\in\mathfrak{F}\setminus(\{\F_\infty\}\cup \mathfrak{F}')$. Lemma~\ref{l:wo2} implies that $A_F\subseteq \{2,p,q\}$, $\Pi_F\subseteq \{2\}$ and 
	$\alpha_E{\restriction}A_F\setminus\Pi_E=\alpha_F{\restriction}A_F\setminus\Pi_E$. Taking into account that $\F_F\notin\mathfrak F'\cup\{\F_\infty\}$ and $\Pi_F\subseteq\{2\}$, we can apply Lemmas~\ref{l:max}, \ref{l:2max} and conclude that $A_F=\{2,p,q\}$. We therefore know that $A_F=A_E$, $\Pi_E\cup\{2\}=\Pi_F\cup\{2\}$ and $\alpha_F{\restriction}A_E\setminus\Pi_F=\alpha_E{\restriction}A_E\setminus\Pi_F$. By Lemma~\ref{l:wo2}, $\F_E=\F_F$ and hence  $\F_E\in\mathfrak{F}''$. 
\end{proof}

\begin{lemma}\label{l:p2p-fix} For any homeomorphism $h$ of the Kirch space and any odd prime number $p$ we have $$\tilde{h}(\F_{\{p,2p\}})=\F_{\{p,2p\}}.$$
%,\;\;\tilde{h}(\F_{\{1,p,2p\}})=\F_{\{1,p,2p\}}\quad\mbox{and}\quad \tilde{h}(\F_{\{2,p,2p\}})=\F_{\{2,p,2p\}}.$$
\end{lemma}

\begin{proof} By Proposition~\ref{p:iso}, the homeomorphism $h$ induces an order isomorphism $\tilde h$ of the superconnecting poset $\mathfrak F$ on the Kirch space. Then $\tilde h[\mathfrak F']=\mathfrak F'$ and $\tilde h[\mathfrak F'']=\mathfrak F''$. 

By Lemmas~\ref{efbis} and \ref{l:realization}, $\mathfrak F''=\mathfrak F''_2\cup\mathfrak F''_3$ where 
$$
\begin{aligned}
\mathfrak F''_2&=\big\{\F_{\{p,2p\}}:p\in\Pi\setminus\{2\}\big\}\quad\mbox{and}\\
\mathfrak F''_3&=\big\{\F_{\{x,pq,2pq\}}:p,q\in\Pi\setminus\{3\},\;p\ne q,\;x\in\{0,\dots,pq-1\}\setminus(p\IZ\cup q\IZ)\big\}.
\end{aligned}
$$
By Lemmas~\ref{l:wo2} and \ref{l:max}, for every filter $\F_{\{p,2p\}}\in\mathfrak F''_2$ the set 
${\uparrow}\F_{\{p,2p\}}=\{\F\in \mathfrak F':\F_{\{p,2p\}}\subset\F_E\}$ coincides with the set $\{\F_{\{a,p,2p\}}:a\in\{1,\dots,p-1\}\}$ and hence has cardinality $p-1$.

On the other hand, for any filter $\F_{\{x,pq,2pq\}}\in\mathfrak F''_3$, the set ${\uparrow}\F_{\{x,pq,2pq\}}=\{\F\in \mathfrak F':\F_{\{x,pq,2pq\}}\subset\F\}$ coincides with the doubleton $\{\F_{\{x,p,2p\}},\F_{\{x,q,2q\}}\}$. 

These order properties uniquely determine the filters $\F_{\{p,2p\}}$ for $p\in\Pi\setminus\{3\}$ and ensure that $\tilde h(\F_{\{p,2p\}})=\F_{\{p,2p\}}$ for every $p\in\Pi\setminus\{3\}$. 

Next, observe that $\F_{\{3,6\}}$ is a unique element $\F$ of $\mathfrak F''$ such that ${\uparrow}\F\cap\bigcup_{p\in\Pi\setminus\{3\}}{\uparrow}\F_{\{p,2p\}}=\emptyset$. This uniqueness order property of $\F_{\{3,6\}}$ ensures that $\tilde h(\F_{\{3,6\}})=\F_{\{3,6\}}$.
\end{proof}

Lemmas~\ref{l:2fix} and \ref{l:p2p-fix} imply

\begin{lemma}\label{l:12fix} For any homeomorphism $h$ of the Kirch space and any odd prime number $p$ we have $$\tilde{h}(\F_{\{1,p,2p\}})=\F_{\{1,p,2p\}}\quad\mbox{and}\quad \tilde{h}(\F_{\{2,p,2p\}})=\F_{\{2,p,2p\}}.$$
\end{lemma}

\begin{lemma}\label{ppix} For an integer number $x\ge 3$ and an odd prime $p$, the following conditions are equivalent:
\begin{enumerate}
\item $p\in\Pi_x$;
\item $\F_{\{1,x\}}\subseteq \F_{\{1,p,2p\}}$ and $\F_{\{2,x\}}\subseteq \F_{\{2,p,2p\}}$.
\end{enumerate}
\end{lemma}

\begin{proof} If $p\in\Pi_x$, then $A_{\{1,p,2p\}}=\{2,p\}\subseteq A_{\{1,x\}}$, $\Pi_{\{1,x\}}=\emptyset=\Pi_{\{1,p,2p\}}$  and $\alpha_{\{1,x\}}(p)=1=\alpha_{\{1,p,2p\}}(p)$. By Lemma~\ref{l:wo2}, $\F_{\{1,x\}}\subseteq\F_{\{1,p,2p\}}$. By analogy we can prove that $\F_{\{2,x\}}\subseteq \F_{\{2,p,2p\}}$.

Conversely, assume $\F_{\{1,x\}}\subseteq \F_{\{1,p,2p\}}$ and $\F_{\{2,x\}}\subseteq \F_{\{2,p,2p\}}$. By Lemmas~\ref{l:wo2} and \ref{2ae}, we have
$$\{2,p\}=A_{\{1,p,2p\}}\subseteq A_{\{1,x\}}=\Pi_{x}\cup\Pi_{x-1}
\text{ and }\{2,p\}=A_{\{2,p,2p\}}\subseteq A_{\{2,x\}}=\{2\}\cup\Pi_x\cup\Pi_{x-2}$$
and hence $p\in (\Pi_x\cup\Pi_{x-1})\cap(\Pi_x\cap\Pi_{x-2})\setminus\{2\}\subseteq \Pi_x$.
\end{proof}

%\begin{lemma}\label{lst1} For every homeomorphism $h$ of the Kirch space and every $x\ge 3$ we have $\Pi_x\setminus\{2\}=\Pi_{h(x)}\setminus\{2\}$.
%\end{lemma}

Proposition~\ref{p:iso} and Lemmas~\ref{l:2fix}, \ref{l:12fix}, \ref{ppix} imply

\begin{lemma}\label{Pix} For every homeomorphism $h$ of the Kirch space and any number $x\in\IN$ we have $$\Pi_x\cup\{2\}=\Pi_{h(x)}\cup\{2\}.$$ 
\end{lemma}

%\begin{proof}
%{\color{red} Fix $n\in \IN$ and consider $x=pn$. 
%Then 	$\F_{\{1,x\}}\subseteq \F_{\{1,p,2p\}}$ and $\F_{\{2,x\}}\subseteq \F_{\{2,p,2p\}}$
%	by  Lemma~\ref{ppix}.   
%Lemmas 13 and 14 imply that
%	$$\F_{\{1,h(x)\}}=\widetilde{h}(\F_{\{1,x\}})\subseteq \widetilde{h}(\F_{\{1,p,2p\}})=\F_{\{1,p,2p\}}$$
%	and 
%$$\F_{\{2,h(x)\}}=\widetilde{h}(\F_{\{2,x\}})\subseteq \widetilde{h}(\F_{\{2,p,2p\}})=\F_{\{2,p,2p\}}.$$
%Using again Lemma~\ref{ppix} we obtain that
%$p\in\Pi_{h(x)}$ what means $h(x)\in p\IN$. 
%Thus $h(p\IN)\subseteq p\IN$ for each homeomorphism $h$ of the Kirch space. So, $h^{-1}(p\IN)\subseteq \IN$, and this implies that $h (p\IN) = p\IN$ for every  homeomorphism of the Kirch space.
%}\end{proof}

For every prime number $p$ consider the set $$V_p=\{2^{n-1}p^m:n,m\in\IN\}$$ of numbers $x\in\IN$ such that $p\in\Pi_x\subseteq \{2,p\}$. Lemmas~\ref{l:2fix} and \ref{Pix} imply that $h[V_p]=V_p$ for every homeomorphism $h$ of the Kirch space.

Consider the graph $\Gamma_p=(V_p,\mathcal E_p)$ on the set $V_p$ with the set of edges $$\mathcal E_p:=\big\{E\in[V_p]^2:A_E=\{2,p\}\big\}.$$

\begin{lemma}\label{l:graph} For every prime number $p$ and every homeomorphism $h$ of the Kirch space, the restriction of $h$ to $V_p$ is an isomorphism of the graph $\Gamma_p$.
\end{lemma}

\begin{proof} Let $E\in \mathcal{E}_p$. Since $p\in \Pi_E$, we can apply Lemma~\ref{efbis} and conclude that $\F_E\in\mathfrak{F}''$.  Using fact that $\tilde{h}$ is an order isomorphism of $\mathfrak{F}$ we get $\F_{h[E]}=\tilde{h}(\F_E)\in\mathfrak{F}''$. Since $h[E]\subseteq h[V_p]=V_p$, we obtain $p\in \Pi_{h[E]}$. Using Lemma~\ref{efbis} once more, we obtain that $A_{h[E]}=\{2,p\}$, which means that $h[E]\in \mathcal{E}_p$. By analogical reasoning we can prove that $h^{-1}[E]\in \mathcal{E}_p$ for every $E\in \mathcal{E}_p$. This means that $h{\restriction}V_p$ is isomorphism of the graph $\Gamma_p$.
	\end{proof}

The structure of the graph $\Gamma_p$ depends on properties of the prime number $p$. 

A prime number $p$ is called
\begin{itemize}
\item {\em Fermat prime} if $p=2^n+1$ for some $n\in\IN$;
\item {\em Mersenne prime} if $p=2^n-1$ for some $n\in\IN$;
\item {\em Fermat--Mersenne} if $p$ is Fermat prime or Mersenne prime.
\end{itemize}
It is known (and easy to see) that for any Fermat prime number $p=2^n+1$ the exponent $n$ is a power of $2$, and for any Mersenne prime number $p=2^n-1$ the power $n$ is a prime number. It is not known whether there are infinitely many Fermat--Mersenne prime numbers. All known Fermat prime numbers are the numbers $2^{2^n}+1$ for $0\le n\le 4$ (see {\tt oeis.org/A019434} in \cite{OEIS}). At the moment only 51 Mersenne prime numbers are known, see the sequence {\tt oeis.org/A000043} in \cite{OEIS}.

\begin{lemma}\label{struct} Let $p$ be an odd prime number.
\begin{enumerate}
\item If $p=3$, then the set $\mathcal E_p$ of edges of the graph $\Gamma_p$ coincides with the set of doubletons\\
$\{2^{a-1}3^b,2^{a-1}3^{b+1}\}$, $\{2^{a-1}3^b,2^{a-1}3^{b+2}\}$, $\{2^{a-1}3^b,2^{a}3^{b}\}$, $\{2^{a-1}3^b,2^{a+1}3^{b}\}$, $\{2^{a-1}3^{b+1},2^{a+1}3^b\}$,\\
$\{2^{a+1}3^{b},2^{a}3^{b+1}\}$, $\{2^{a+3}3^b,2^a3^{b+2}\}$
for some $a,b\in\IN$.
\item If $p=2^m+1>3$ is Fermat prime, then\newline $\mathcal E_p=\big\{\{2^{a-1}p^b,2^{a-1}p^{b+1}\},\{2^{a-1}p^b,2^{a}p^b\},\{2^{m+a-1}p^{b},2^{a-1}p^{b+1}\} :a,b\in\IN\big\}$.
\item 
If $p=2^m-1>3$ is Mersenne prime, then\newline $\mathcal E_p=\big\{\{2^{a-1}p^b,2^{a}p^b\},\{2^{a-1}p^b,2^{m+a-1}p^b\}, \{2^{a-1}p^{b+1},2^{m+a-1}p^b\}:a,b\in\IN\big\}$.
\item If $p$ is not Fermat--Mersenne, then $\mathcal E_p=\big\{\{2^ap^b,2^{a-1}p^b\}:a,b\in\IN\big\}$.
\end{enumerate} 
\end{lemma}

\begin{proof} 	Proof of Lemma~\ref{struct} in each of cases (1)--(4) will be similar. Edges of the graph $\Gamma_p$ are  $2$-element subsets of the set $V_p$ such that $A_E=\{2,p\}$. Since vertices of the graph $\Gamma_p$ are numbers of the form $2^{n-1}p^m$, where $n,m\in\IN$, we can apply Lemma~\ref{2ae} and conclude that a doubleton $\{x,y\}\subset V_p$ belongs to $\mathcal E_p$ if and only if $\{2,p\}=\{2\}\cup \Pi_x\cup\Pi_y\cup\Pi_{x-y}$. In subsequent proofs, we will intensively use the Mih\u ailescu Theorem~\ref{Mihailescu} saying that $2^3,3^2$ is a unique pair of consecutive powers. 
\smallskip	
	
1. First we consider the case of $p=3$. It is easy to see that the doubletons $\{x,y\}$ written in the statement (1) have $\Pi_x\cup\Pi_y\cup\Pi_{x-y}\subseteq\{2,3\}$, which implies that $\{x,y\}\in\mathcal E_3$. It remains to show that every doubleton $\{x,y\}\in\mathcal E_3$ is of the form indicated in the statement (1). Write $\{x,y\}$ as $\{2^{a-1}3^b,2^{c-1}3^d\}$ for some $a,b,c,d\in\IN$ such that $2^{a-1}3^b<2^{c-1}3^d$. 

If $a=c$, then $b<d$ and the inclusion $\Pi_{x-y}\subseteq\{2,3\}$ implies that $\Pi_{3^{d-b}-1}\subseteq\{2,3\}$ and hence $3^{d-b}-1$ is a power of $2$. By the Mih\u ailescu Theorem~\ref{Mihailescu}, $d-b\in\{1,2\}$, which means that $\{x,y\}$ is equal to $\{2^{a-1}3^b,2^{a-1}3^{b+1}\}$ or $\{2^{a-1}3^b,2^{a-1}3^{b+2}\}$.

If $b=d$, then $a<c$ and the inclusion  $\Pi_{x-y}\subseteq\{2,3\}$ implies that $\Pi_{2^{c-a}-1}\subseteq\{2,3\}$ and hence $2^{c-a}-1$ is a power of $3$. By the Mih\u ailescu Theorem~\ref{Mihailescu}, $c-a\in\{1,2\}$, which means that $\{x,y\}$ is equal to $\{2^{a-1}3^b,2^{a}3^{b}\}$ or $\{2^{a-1}3^b,2^{a+1}3^{b}\}$.

So, we assume that $a\ne c$ and $b\ne d$. In this case we should consider four subcases.

If $a<c$ and $b<d$, then $\Pi_{x-y}\subseteq\{2,3\}$ implies that each prime divisor of $2^{c-a}3^{d-b}-1$ is equal to $2$ or $3$, which is not possible.

If  $a<c$ and $b>d$, then $\Pi_{x-y}\subseteq\{2,3\}$ and $2^{a-1}3^b<2^{c-1}3^d$ imply that $2^{c-a}-3^{b-d}=1$ and hence $c-a=2$ and $b-d=1$ by the Mih\u ailescu Theorem~\ref{Mihailescu}.  In this case $\{x,y\}=\{2^{a-1}3^{d+1},2^{a+1}3^d\}$.

If $a>c$ and $b<d$, then  $\Pi_{x-y}\subseteq\{2,3\}$ and $2^{a-1}3^b<2^{c-1}3^d$ imply that $3^{d-b}-2^{a-c}=1$ and hence $(d-b,a-c)\in\{(1,1),(2,3)\}$ by the Mih\u ailescu Theorem~\ref{Mihailescu}.  In this case $\{x,y\}$ is equal to $\{2^{c+1}3^{b},2^{c}3^{b+1}\}$ or $\{2^{c+3}3^b,2^c3^{b+2}\}$.

The subcase $a>c$ and $b>d$ is forbidden by the inequality $2^{a-1}3^b<2^{c-1}3^d$.
\smallskip

2. Assume that  $p=2^m+1>3$ is a Fermat prime. In this case $m>1$.	It is easy to check that every doubleton $\{x,y\}\in\big\{\{2^{a-1}p^b,2^{a-1}p^{b+1}\},\{2^{a-1}p^b,2^{a}p^b\},\{2^{m+a-1}p^{b},2^{a-1}p^{b+1}\} :a,b\in\IN\big\}$ has $A_{\{x,y\}}=\{2\}\cup \Pi_x\cup\Pi_y\cup\Pi_{x-y}=\{2,p\}$ and hence $\{x,y\}\in\mathcal E_p$. 

Now assume that $\{x,y\}\in\mathcal E_p$ is an edge of the graph $\Gamma_p$. Then $\{2\}\cup\Pi_x\cup\Pi_y\cup\Pi_{x-y}=A_{\{x,y\}}=\{2,p\}$ and $\{x,y\}$ can be written as $\{2^{a-1}p^b,2^{c-1}p^d\}$ for some $a,b,c,d\in\IN$ with $2^{a-1}p^b<2^{c-1}p^d$.

If $a=c$, then $b<d$ and the inclusion $\Pi_{x-y}\subseteq\{2,p\}$ implies that $\Pi_{p^{d-b}-1}\subseteq\{2,p\}$ and hence $p^{d-b}-1$ is a power of $2$. By the Mih\u ailescu Theorem~\ref{Mihailescu}, $d-b=1$, which means that $\{x,y\}$ is equal to $\{2^{a-1}p^b,2^{a-1}p^{b+1}\}$.

If $b=d$, then $a<c$ and the inclusion  $\Pi_{x-y}\subseteq\{2,p\}$ implies that $\Pi_{2^{c-a}-1}\subseteq\{2,p\}$ and hence $2^{c-a}-1$ is a power of $p$. By the Mih\u ailescu Theorem~\ref{Mihailescu}, $2^{c-a}-1\in\{1,p\}=\{1,2^m+1\}$ and hence $c-a=1$, which means that $\{x,y\}$ is equal to $\{2^{a-1}p^b,2^{a}p^{b}\}$.

So, we assume that $a\ne c$ and $b\ne d$. By analogy with the case of $p=3$, we can show that the subcases ($a<c$ and $b<d$) and ($a>c$ and $b>d$) are impossible.

If $a<c$ and $b>d$, then $\Pi_{x-y}\subseteq\{2,p\}$ implies that $2^{c-a}-p^{b-d}=1$. In this case the Mih\u ailescu Theorem~\ref{Mihailescu} ensures that $b-d=1$ and hence $2^{c-a}=p+1=2^m+2$ which is not possible (as $m>1$).

If $a>c$ and $b<d$, then $\Pi_{x-y}\subseteq\{2,p\}$ implies that $p^{d-a}-2^{a-c}=1$. In this case the Mih\u ailescu Theorem~\ref{Mihailescu} implies that $d-b=1$ and hence $2^{a-c}=p-1=2^m$ and $a-c=m$. In this case $\{x,y\}=\{2^{c+m}2^b,2^cp^{b+1}\}$.
\smallskip

3. Assume that  $p=2^m-1>3$ is a Mersenne prime. In this case $m>2$.	It is easy to check that every doubleton $\{x,y\}\in\big\{\{2^ap^b,2^{a-1}p^b\},\{2^{a-1}p^b,2^{m+a-1}p^b\}, \{2^{a-1}p^{b+1},2^{m+a-1}p^b\} :a,b\in\IN\big\}$ has $A_{\{x,y\}}=\{2\}\cup \Pi_x\cup\Pi_y\cup\Pi_{x-y}=\{2,p\}$ and hence $\{x,y\}\in\mathcal E_p$. 

Now assume that $\{x,y\}\in\mathcal E_p$ is an edge of the graph $\Gamma_p$. Then $\{2\}\cup\Pi_x\cup\Pi_y\cup\Pi_{x-y}=A_{\{x,y\}}=\{2,p\}$ and $\{x,y\}$ can be written as $\{2^{a-1}p^b,2^{c-1}p^d\}$ for some $a,b,c,d\in\IN$ with $2^{a-1}p^b<2^{c-1}p^d$.

If $a=c$, then $b<d$ and the inclusion $\Pi_{x-y}\subseteq\{2,p\}$ implies that $\Pi_{p^{d-b}-1}\subseteq\{2,p\}$ and hence $p^{d-b}-1$ is a power of $2$. By the Mih\u ailescu Theorem~\ref{Mihailescu}, $d-b=1$ and hence $2^m-2=p-1$ is a power of $2$, which is not true as $m>2$.

If $b=d$, then $a<c$ and the inclusion  $\Pi_{x-y}\subseteq\{2,p\}$ implies that $\Pi_{2^{c-a}-1}\subseteq\{2,p\}$ and hence $2^{c-a}-1$ is a power of $p$. By the Mih\u ailescu Theorem~\ref{Mihailescu}, $2^{c-a}-1\in\{1,p\}=\{1,2^m-1\}$ and hence $c-a\in\{1,m\}$, which means that $\{x,y\}$ is equal to $\{2^{a-1}p^b,2^{a}p^{b}\}$ or $\{2^{a-1}p^b,2^{m+a-1}p^b\}$.

So, we assume that $a\ne c$ and $b\ne d$. By analogy with the case of $p=3$, we can show that the subcases ($a<c$ and $b<d$) and ($a>c$ and $b>d$) are impossible. 	

If $a<c$ and $b>d$, then $\Pi_{x-y}\subseteq\{2,p\}$ implies that $2^{c-a}-p^{b-d}=1$. In this case the Mih\u ailescu Theorem~\ref{Mihailescu} ensures that $b-d=1$ and hence $2^{c-a}=p+1=2^m$ and $c-a=m$. In this case $\{x,y\}=\{2^{a-1}p^{d+1},2^{m+a-1}p^d\}$.

If $a>c$ and $b<d$, then $\Pi_{x-y}\subseteq\{2,p\}$ implies that $p^{d-a}-2^{a-c}=1$. In this case the Mih\u ailescu Theorem~\ref{Mihailescu} implies that $d-b=1$ and hence $2^{a-c}=p-1=2^m-2$, which is not possible as $m>2$. 
\smallskip

4. Assume that $p$ is not Fermat-Mersennne. It is easy to check that every doubleton $\{x,y\}\in\big\{\{2^{a-1}p^b,2^{a-1}p^{b+1}\}:a,b\in\IN\big\}$ has $A_{\{x,y\}}=\{2\}\cup \Pi_x\cup\Pi_y\cup\Pi_{x-y}=\{2,p\}$ and hence $\{x,y\}\in\mathcal E_p$. 

Now assume that $\{x,y\}\in\mathcal E_p$ is an edge of the graph $\Gamma_p$. Then $\{2\}\cup\Pi_x\cup\Pi_y\cup\Pi_{x-y}=A_{\{x,y\}}=\{2,p\}$ and $\{x,y\}$ can be written as $\{2^{a-1}p^b,2^{c-1}p^d\}$ for some $a,b,c,d\in\IN$ with $2^{a-1}p^b<2^{c-1}p^d$.

If $a=c$, then $b<d$ and the inclusion $\Pi_{x-y}\subseteq\{2,p\}$ implies that $\Pi_{p^{d-b}-1}\subseteq\{2,p\}$ and hence $p^{d-b}-1$ is a power of $2$. By the Mih\u ailescu Theorem~\ref{Mihailescu}, $d-b=1$ and hence $p$ is a Fermat prime, which is not true.

If $b=d$, then $a<c$ and the inclusion  $\Pi_{x-y}\subseteq\{2,p\}$ implies that $\Pi_{2^{c-a}-1}\subseteq\{2,p\}$ and hence $2^{c-a}-1$ is a power of $p$. By the Mih\u ailescu Theorem~\ref{Mihailescu}, $2^{c-a}-1\in\{1,p\}$. Taking into account that $p$ is not Mersenne prime,  we conclude that $2^{c-a}-1=1$ and hence $c-a=1$. Then $\{x,y\}=\{2^{a-1}p^b,2^{a}p^{b}\}$.

So, we assume that $a\ne c$ and $b\ne d$. By analogy with the case of $p=3$, we can show that the subcases ($a<c$ and $b<d$) and ($a>c$ and $b>d$) are impossible. 	

If $a<c$ and $b>d$, then $\Pi_{x-y}\subseteq\{2,p\}$ implies that $2^{c-a}-p^{b-d}=1$. In this case the Mih\u ailescu Theorem~\ref{Mihailescu} ensures that $b-d=1$ and hence $p=2^{c-a}-1$ is a Mersenne prime, which is not true.

If $a>c$ and $b<d$, then $\Pi_{x-y}\subseteq\{2,p\}$ implies that $p^{d-a}-2^{a-c}=1$. In this case the Mih\u ailescu Theorem~\ref{Mihailescu} implies that $d-b=1$ and hence $p=1+2^{a-c}$ is a Fermat prime, which is not true.	
\end{proof}

In the following diagrams we draw the graphs $\Gamma_p$ for $p$ equal to $3,5,7,11$. Observe that $3$ is both Fermat and Mersenne prime, $5$ is Fermat prime, $7$ is Mersenne prime and $11$ is not Fermat--Mersenne.

$$
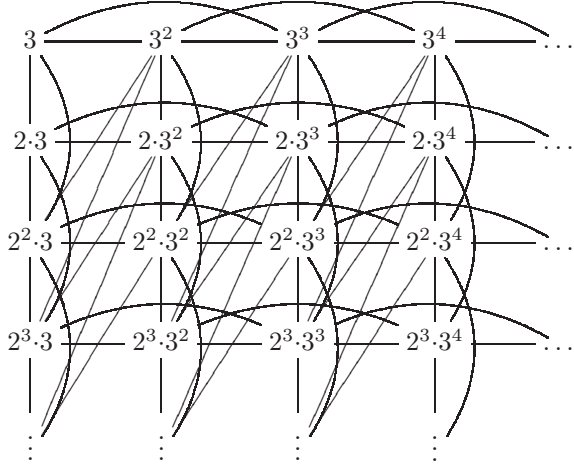
\xymatrix{
3\ar@{-}[d]\ar@{-}@/^15pt/[dd]\ar@{-}@/^15pt/[rr]\ar@{-}[r]&3^2\ar@{-}[ddl]\ar@{-}[r]\ar@{-}[d]\ar@{-}@/^15pt/[rr]\ar@{-}[dddl]\ar@{-}@/^15pt/[dd]&3^3\ar@{-}[ddl]\ar@{-}@/^15pt/[dd]\ar@{-}[r]\ar@{-}[d]\ar@{-}@/^15pt/[rr]\ar@{-}[dddl]&3^4\ar@{-}[ddl]\ar@{-}@/^15pt/[dd]\ar@{-}[r]\ar@{-}[d]\ar@{-}[dddl]&\dots\\
2{\cdot}3\ar@{-}@/^15pt/[dd]\ar@{-}[d]\ar@{-}@/^15pt/[rr]\ar@{-}[r]&2{\cdot}3^2\ar@{-}[ddl]\ar@{-}@/^15pt/[dd]\ar@{-}[r]\ar@{-}[d]\ar@{-}@/^15pt/[rr]\ar@{-}[dddl]&2{\cdot}3^3\ar@{-}[ddl]\ar@{-}@/^15pt/[dd]\ar@{-}[r]\ar@{-}[d]\ar@{-}@/^15pt/[rr]\ar@{-}[dddl]&2{\cdot}3^4\ar@{-}[ddl]\ar@{-}@/^15pt/[dd]\ar@{-}[r]\ar@{-}[d]\ar@{-}[dddl]&\dots\\
2^2{\cdot}3\ar@{-}@/^15pt/[dd]\ar@{-}[d]\ar@{-}@/^15pt/[rr]\ar@{-}[r]&2^2{\cdot}3^2\ar@{-}[ddl]\ar@{-}@/^15pt/[dd]\ar@{-}[r]\ar@{-}[d]\ar@{-}@/^15pt/[rr]&2^2{\cdot}3^3\ar@{-}[ddl]\ar@{-}@/^15pt/[dd]\ar@{-}[r]\ar@{-}[d]\ar@{-}@/^15pt/[rr]&2^2{\cdot}3^4\ar@{-}[ddl]\ar@{-}@/^15pt/[dd]\ar@{-}[r]\ar@{-}[d]&\dots\\
2^3{\cdot}3\ar@{-}[d]\ar@{-}@/^15pt/[rr]\ar@{-}[r]&2^3{\cdot}3^2\ar@{-}[r]\ar@{-}[d]\ar@{-}@/^15pt/[rr]&2^3{\cdot}3^3\ar@{-}[r]\ar@{-}[d]\ar@{-}@/^15pt/[rr]&2^3{\cdot}3^4\ar@{-}[r]\ar@{-}[d]&\dots\\
\vdots&\vdots&\vdots&\vdots\\
}\hskip30pt
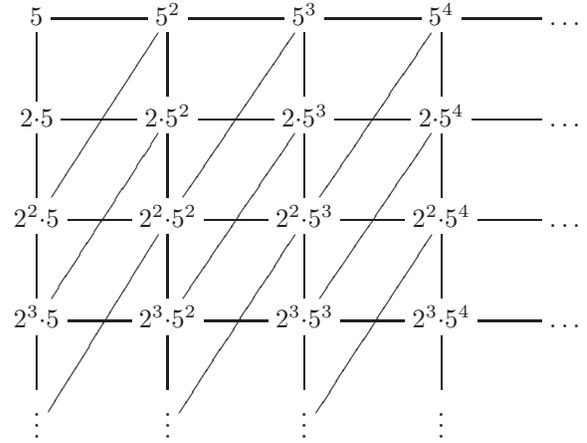
\xymatrix{
5\ar@{-}[r]\ar@{-}[d]&5^2\ar@{-}[r]\ar@{-}[d]\ar@{-}[ldd]&5^3\ar@{-}[r]\ar@{-}[d]\ar@{-}[ldd]&5^4\ar@{-}[r]\ar@{-}[d]\ar@{-}[ldd]&\dots\\
2{\cdot}5\ar@{-}[r]\ar@{-}[d]&2{\cdot}5^2\ar@{-}[r]\ar@{-}[d]\ar@{-}[ldd]&2{\cdot}5^3\ar@{-}[r]\ar@{-}[d]\ar@{-}[ldd]&2{\cdot}5^4\ar@{-}[r]\ar@{-}[d]\ar@{-}[ldd]&\dots\\
2^2{\cdot}5\ar@{-}[r]\ar@{-}[d]&2^2{\cdot}5^2\ar@{-}[r]\ar@{-}[d]\ar@{-}[ldd]&2^2{\cdot}5^3\ar@{-}[r]\ar@{-}[d]\ar@{-}[ldd]&2^2{\cdot}5^4\ar@{-}[r]\ar@{-}[d]\ar@{-}[ldd]&\dots\\
2^3{\cdot}5\ar@{-}[r]\ar@{-}[d]&2^3{\cdot}5^2\ar@{-}[r]\ar@{-}[d]&2^3{\cdot}5^3\ar@{-}[r]\ar@{-}[d]&2^3{\cdot}5^4\ar@{-}[r]\ar@{-}[d]&\dots\\
\vdots&\vdots&\vdots&\vdots\\
&&&&
}
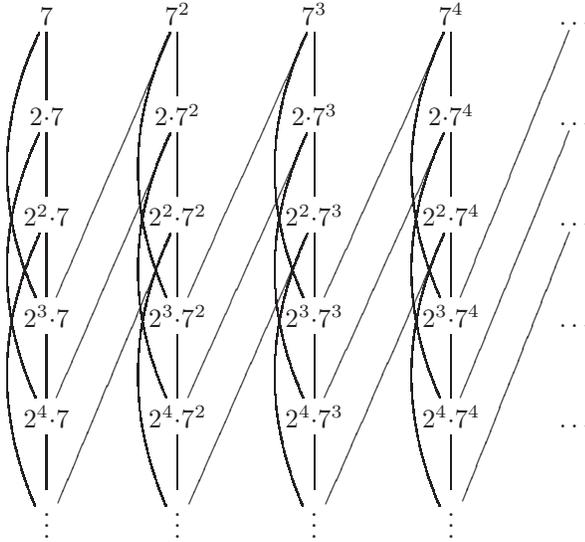
$$

$$
\xymatrix{
7\ar@{-}[d]\ar@{-}@/_15pt/[ddd]&7^2\ar@{-}[d]\ar@{-}@/_15pt/[ddd]\ar@{-}[dddl]&7^3\ar@{-}[d]\ar@{-}@/_15pt/[ddd]\ar@{-}[dddl]&7^4\ar@{-}[d]\ar@{-}@/_15pt/[ddd]\ar@{-}[dddl]&\dots\ar@{-}[dddl]\\
2{\cdot}7\ar@{-}[d]\ar@{-}@/_15pt/[ddd]&2{\cdot}7^2\ar@{-}[d]\ar@{-}@/_15pt/[ddd]\ar@{-}[dddl]&2{\cdot}7^3\ar@{-}[d]\ar@{-}@/_15pt/[ddd]\ar@{-}[dddl]&2{\cdot}7^4\ar@{-}[d]\ar@{-}@/_15pt/[ddd]\ar@{-}[dddl]&\dots\ar@{-}[dddl]\\
2^2{\cdot}7\ar@{-}[d]\ar@{-}@/_15pt/[ddd]&2^2{\cdot}7^2\ar@{-}[d]\ar@{-}@/_15pt/[ddd]\ar@{-}[dddl]&2^2{\cdot}7^3\ar@{-}[d]\ar@{-}@/_15pt/[ddd]\ar@{-}[dddl]&2^2{\cdot}7^4\ar@{-}[d]\ar@{-}@/_15pt/[ddd]\ar@{-}[dddl]&\dots\ar@{-}[dddl]\\
2^3{\cdot}7\ar@{-}[d]&2^3{\cdot}7^2\ar@{-}[d]&2^3{\cdot}7^3\ar@{-}[d]&2^3{\cdot}7^4\ar@{-}[d]&\dots\\
2^4{\cdot}7\ar@{-}[d]&2^4{\cdot}7^2\ar@{-}[d]&2^4{\cdot}7^3\ar@{-}[d]&2^4{\cdot}7^4\ar@{-}[d]&\dots\\
\vdots&\vdots&\vdots&\vdots\\
}\hskip30pt
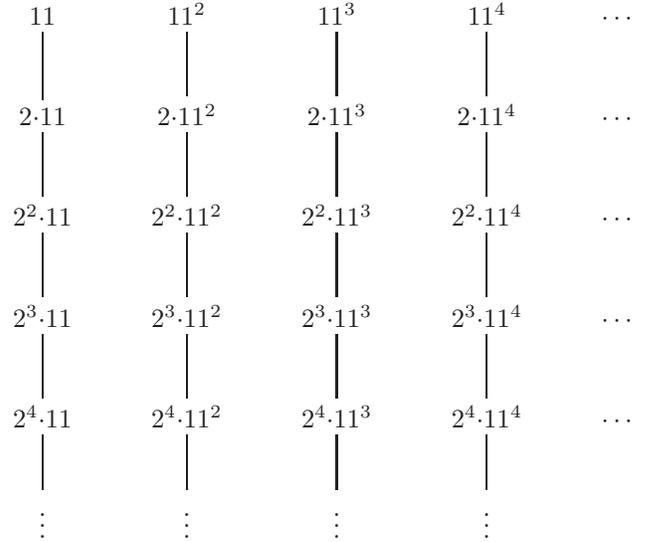
\xymatrix{
11\ar@{-}[d]&11^2\ar@{-}[d]&11^3\ar@{-}[d]&11^4\ar@{-}[d]&\cdots\\
2{\cdot}11\ar@{-}[d]&2{\cdot}11^2\ar@{-}[d]&2{\cdot}11^3\ar@{-}[d]&2{\cdot}11^4\ar@{-}[d]&\cdots\\
2^2{\cdot}11\ar@{-}[d]&2^2{\cdot}11^2\ar@{-}[d]&2^2{\cdot}11^3\ar@{-}[d]&2^2{\cdot}11^4\ar@{-}[d]&\cdots\\
2^3{\cdot}11\ar@{-}[d]&2^3{\cdot}11^2\ar@{-}[d]&2^3{\cdot}11^3\ar@{-}[d]&2^3{\cdot}11^4\ar@{-}[d]&\cdots\\
2^4{\cdot}11\ar@{-}[d]&2^4{\cdot}11^2\ar@{-}[d]&2^4{\cdot}11^3\ar@{-}[d]&2^4{\cdot}11^4\ar@{-}[d]&\cdots\\
\vdots&\vdots&\vdots&\vdots
}
$$

\begin{lemma}\label{l:graph} Let $p$ be an odd prime number and $h$ be a homeomorphism of the Kirch space.
\begin{enumerate}
\item If $p$ is Fermat-Mersenne, then $h(p)=p$;
\item If $p$ is not Fermai-Mersenne, then $h[p^\IN]=p^\IN$.
\end{enumerate}
\end{lemma}

\begin{proof} Given an odd prime number $p$, consider the graph $\Gamma_p=(V_p,\mathcal E_p)$. 

First we consider the case $p=3$. In this case Lemma~\ref{struct}(1) ensures that the degree of the vertex $3$ in the graph $\Gamma_3$ is equal to $4$ but the other vertices have  degree at least $5$. Hence $h(3)=3$. 	

Next, we assume that  $p>3$ is Fermat--Mersenne prime. In this case Lemma~\ref{struct}(2,3) implies that the degree of the vertex $p$ in the graph $\Gamma_p$ is $2$ but the other vertices have  degree at least $3$. Hence $h(p)=p$.

Finally, assume that $p$ is not Fermat-Mersenne. Then Lemma~\ref{struct}(4) ensures that the set $p^\IN$ coincides with the set of vertices of order 1 in the graph $\Gamma_p$. Taking into account that $h{\restriction}V_p$ is an isomorphism of the graph $\Gamma_p$, we conclude that $h[p^\IN]=p^\IN$.
\end{proof}

To prove that $h(p)=p$ for any prime number $p$, we  will need the following lemma.

\begin{lemma}\label{l:F1x} For any integer number $x\in\IN\setminus\{1\}$ the filter $\F_{\{1,x\}}$ is the greatest element of the subset $$\mathfrak F_x=\{\F_{\{1,x^n\}}:n\in\IN\}$$ of the poset $\mathfrak F$. If $x\notin \{2m:m\in\IN\}\cup\{2^m-1:m\in\IN\}$, then $\{n\in\IN:\F_{\{1,x^n\}}=\F_{\{1,x\}}\}=\{1\}$.
\end{lemma}

\begin{proof} Observe that for every $n\in\IN$ the number $x-1$ divides $x^n-1$, which implies $$A_{\{1,x\}}=\{2\}\cup\Pi_x\cup\Pi_{x-1}\subseteq\{2\}\cup\Pi_{x^n}\cup\Pi_{x^n-1}=A_{\{1,x^n\}}.$$Observe also that $\Pi_{\{1,x\}}=\emptyset=\Pi_{\{1,x^n\}}$ and $\alpha_{\{1,x\}}(p)=1=\alpha_{\{1,x^n\}}(p)$ for every $p\in A_{\{1,x\}}$. By Lemma~\ref{l:wo2}, $\F_{\{1,x^n\}}\subseteq\F_{\{1,x\}}$, which means that $\F_{\{1,x\}}$ is the largest element of the poset $\mathfrak F_x$.

Now assume that $x\notin\{2m:m\in\IN\}\cup\{2^m-1:n\in\IN\}$ and $\F_{\{1,x\}}=\F_{\{1,x^n\}}$ for some number $n$. We should prove that $n=1$. To derive a contradiction, assume that $n\ge 2$. By Lemmas~\ref{l:wo2} and \ref{2ae}, $$\{2\}\cup\Pi_x\cup\Pi_{x^n-1}=A_{\{1,x^n\}}=A_{\{1,x\}}=\{2\}\cup\Pi_x\cup\Pi_{x-1}$$ and hence $\Pi_{x^n-1}\subseteq\{2\}\cup\Pi_{x-1}=\Pi_{x-1}\subseteq\bigcup_{0<k<n}\Pi_{x^k-1}$. By Zsigmondy Theorem~\ref{Zsigmondy}, $x\in \{2\}\cup\{2^m-1\}_{m\in\IN}$, which contradicts our assumption. 
\end{proof}

\begin{lemma}\label{l:pfix} For any homeomorphism $h$ of the Kirch space and any prime number $p$ we have $h(p)=p$.
\end{lemma}

\begin{proof} If $p=2$, then $h(p)=p$ by Lemma~\ref{l:2fix}. If $p$ is Fermat--Mersenne, then $h(p)=p$ by Lemma~\ref{l:graph}. So, we assume $p$ is not Fermat-Mersenne. By Lemma~\ref{l:graph}, $h[p^\IN]=p^\IN$, which implies $\tilde h[\mathfrak F_p]=\mathfrak F_p$ where
$$\mathfrak F_p=\{\F_{\{1,p^n\}}:n\in\IN\}.$$ By Proposition~\ref{p:iso}, $\tilde h$ induces an order isomorphism of the poset $\mathfrak F_p$ (endowed with the inclusion order, inherited from the poset $\mathfrak F$).

 By Lemma~\ref{l:F1x}, $n=1$ is a unique number such that $\F_{\{1,p^n\}}$ coincides with the greatest element $\F_{\{1,p\}}$ of the poset $\mathfrak F_x$. This order characterization of the filter $\F_{\{1,p\}}$ implies that $h(p)=p$.
 \end{proof}

Our final lemma completes the proof of Theorem~\ref{t:main}.

\begin{lemma} The homeomorphism group of the Kirch space is trivial.
\end{lemma}

\begin{proof} To derive a contradiction, assume that the Kirch space admits a homeomorphism $h$ such that $h(x)\ne x$ for some number $x$. By the Hausdorff property of the Kirch space and the continuity of $h$, there exists a neighborhood $O_x$ of $x$ in the Kirch topology such that $h[O_x]\cap O_x=\emptyset$. By the  Dirichlet Theorem~\ref{Dirichlet}, the open set $O_x$ contains some prime number $p$. Then $h[O_x]\cap O_x=\emptyset$ implies $h(p)\ne p$, which contradicts Lemma~\ref{l:pfix}.
\end{proof}

\section{Acknowledgements}

The authors would like to thank the Mathoverflow users {\tt Aaron~Meyerowitz},  {\tt Fedor~Petrov},\break {\tt Gerhard Paseman}, {\tt Gerry Myerson}, {\tt Ofir~Gorodetsky}, {\tt Paul~Monsky}, {\tt Will~Sawin}, {\tt Wojowu}, {\tt YCor}\break for their help with number-theoretic problems ({\tt mathoverflow.net/q/363703, mathoverflow.net/q/347798,\break mathoverflow.net/q/347774, mathoverflow.net/q/347039, mathoverflow.net/a/363817}) that arised\break during writing this paper.

\end{document}